\documentclass[11pt, a4paper, twoside]{amsart}
\usepackage[utf8]{inputenc}
\usepackage{amsmath}
\usepackage{amsthm}
\usepackage{amssymb}
\usepackage{extpfeil}
\DeclareMathOperator{\Hom}{Hom}
\DeclareMathOperator{\End}{End}

\DeclareMathOperator{\modf}{mod}
\DeclareMathOperator{\add}{add}
\DeclareMathOperator{\Ob}{Ob}
\DeclareMathOperator{\Ker}{Ker}
\DeclareMathOperator{\im}{Im}

\DeclareMathOperator{\rad}{Rad}
\DeclareMathOperator{\ind}{Ind}
\DeclareMathOperator{\proj}{proj}
\newcommand{\op}{{^\textit{op}}}
\newcommand{\sm}[1]{\left(\begin{smallmatrix}#1\end{smallmatrix}\right)}
\newcommand{\bm}[1]{\begin{pmatrix}#1\end{pmatrix}}
\newcommand{\HomT}[1]{\Hom_{\mathcal{T}}(T,#1)}
\newtheorem{definition}{Definition}
\newtheorem{theorem}[definition]{Theorem}
\newtheorem{lemma}[definition]{Lemma}
\theoremstyle{definition}
\newtheorem{example}[definition]{Example}
\newtheorem*{setup}{Setup}

\usepackage{tikz}
\usetikzlibrary{calc}
\usetikzlibrary{arrows,decorations.pathmorphing,decorations.pathreplacing, decorations.markings} 

\begin{document}
\title[Abelian quotients of triangulated categories]{Abelian quotients of triangulated categories}
\author{Benedikte Grimeland \and Karin Marie Jacobsen}
\begin{abstract}
We study abelian quotient categories \(\mathcal{A}=\mathcal{T}/\mathcal{J}\), where \(\mathcal{T}\) is a triangulated category and \(\mathcal J\) is an ideal of \(\mathcal{T}\).
Under the assumption that the quotient functor is cohomological we show that it is representable and give an explicit description of the functor. 
We give technical criteria for when a representable functor is a quotient functor, and a criterion for when \(\mathcal{J}\) gives rise to a cluster-tilting subcategory of \(\mathcal{T}\). 
We show that the quotient functor preserves the AR-structure.
As an application we show that if \(\mathcal{T}\) is a finite 2-Calabi-Yau category, then with very few exceptions \(\mathcal J\) is a cluster-tilting subcategory of \(\mathcal{T}\).

\end{abstract}
\maketitle

\section{introduction}

In the literature there are several known methods for forming a triangulated category given an abelian category. Given an abelian category \(\mathcal A\) one can form the homotopy category \(\mathcal{K}(\mathcal A)\) and the derived category \(\mathcal{D}(\mathcal A)\), both of which are triangulated, along with their bounded versions.
Orbit categories \(\mathcal{D}^{b}(\mathcal A)/F\) are known \cite{keller} to be triangulated when \(\mathcal A\) is hereditary and \(F\) is a suitable autoequivalence.
The stable module category of a selfinjective algebra is also triangulated. 

With the introduction of cluster algebras \cite{FZ} and cluster-tilting theory \cite{bmrrt}, cluster-tilting subcategories (or maximal 1-orthogonal subcategories) have been defined, see \cite{iyama}. In \cite{k-zhu}, Koenig and Zhu show that the quotient of any triangulated category by a cluster-tilting subcategory is abelian. However not all triangulated categories contain a cluster-tilting subcategory, but they may still admit an abelian quotient (for an example, see \cite{k-zhu}).
It is also known that for the cluster categories of coherent sheaves on weighted projective lines it is possible to obtain an abelian quotient by factoring out morphisms, without any objects being sent to zero \cite{sheaves} . 

Consider the orbit category \(\mathcal{D}^b(kQ)/\Sigma\), where \(Q\) is a Dynkin diagram and \(\Sigma\) is the suspension functor. This category has the same (finite) number of isomorphism classes of indecomposable objects as \(\modf kQ\), but has a greater number of irreducible morphisms. This motivates us to find out if we can factor out an ideal to obtain an abelian category, possibly without sending any non-zero objects to zero.
Both of the examples mentioned will be revisited in detail in section 4.

Factoring out an ideal from the cluster category of a hereditary algebra has been studied \cite{bmr}. All known abelian quotients of these cluster categories arise from factoring out cluster-tilting subcategories. We show that in the finite case these are in fact all possible abelian quotient categories.

In section 2 we define some notation and show that in the finite case, if an abelian quotient category exists, it has enough projectives.

In section 3 we study a quotient functor from a triangulated category to an abelian category with projective generator. We show that it is representable and naturally equivalent to an explicitly described functor.

Section 4 contains the main result:

\begin{theorem}
\(\HomT{-}\) is a quotient functor from a triangulated category \(\mathcal T\) if and only if the following two criteria are satisfied
\begin{description}
\item[a] For all right minimal morphisms \(T_1\rightarrow T_0\), where \(T_0, T_1\in\add T\), all triangles \(T_1\rightarrow T_0\rightarrow X\xrightarrow h \Sigma T_1\) satisfy \(\HomT h=0\).
\item[b] For all indecomposable \(T\)-supported objects \(X\) there exists a triangle \(T_1\rightarrow T_0\rightarrow X\xrightarrow h \Sigma T_1\) with \(T_1,T_0\in \add T\) and \(\HomT h=0\).
\end{description}
\end{theorem}

In section 5 we show that if it exists, the AR-structure is preserved by the quotient functor.

In section 6 we discuss the special case of triangulated categories with Calabi-Yau dimension 2. We show
\begin{theorem}
Let \(\mathcal T\) be a 2-CY connected triangulated category with finitely many isomorphism classes of indecomposable objects. If \(T\) is an object in \(\mathcal T\) such that \(\HomT -:\mathcal T\rightarrow \modf\Gamma\) is full and dense, then \(T\) is either Schurian or \(T\) a 2-cluster-tilting object in \(\mathcal T\).
\end{theorem}

\section{Background}
\begin{setup}
\(k\) is a field and \(\mathcal{T}\) is a Hom-finite Krull-Schmidt triangulated \(k\)-category. \(\Sigma\) is the suspension functor of \(\mathcal T\).

By \(\mathcal{J}\) we denote an ideal of \(\mathcal{T}\). The quotient category \(\mathcal{T}/\mathcal{J}\) has the same objects as \(\mathcal{T}\), and has morphisms \(\Hom_{\mathcal{T}/\mathcal{J}}(X,Y)=\Hom_{\mathcal{T}}(X,Y)/\mathcal{J}(X,Y)\). 

By construction the projection functor \(\pi:\mathcal{T}\rightarrow \mathcal{T}/\mathcal{J}\) is full and dense. We also assume it to be cohomological.
\end{setup}

Note that the properties of being Hom-finite and a \(k\)-category carries over from \(\mathcal{T}\) to the quotient \(\mathcal{T}/\mathcal{J}\). 
The property of being a Krull-Schmidt category is also inherited by the quotient category. The proof is a slightly adapted version of the proof found in \cite{mi}, taking into account we do not assume that the ideal \(\mathcal J\) always contains objects. 

\begin{lemma}
Let \(\mathcal T\) be a triangulated Krull-Schmidt \(k\)-category, and let \(\mathcal J\) be an ideal in \(\mathcal T\). Then the quotient category \(\mathcal T/\mathcal J\) is also a Krull-Schmidt category. 
\end{lemma}

\begin{proof}
Let \(\overline X\) be an indecomposable non-zero object of \(\mathcal T/\mathcal J\). Then the preimage \(X\) in \(\mathcal T\) can be decomposed into a finite direct sum of indecomposable objects: \(X=\oplus_{i=1}^nX_i\). Let \(e_i: X\xrightarrow{\rho_i}X_i\xrightarrow{\iota_i}X\) be the canonical morphisms in \(\mathcal T\) for \(i\in\left\{1,\ldots,n\right\}\), and denote by \(\overline{e_i}\) the image of \(e_i\) in \(\mathcal T/\mathcal J\). 

Since \(\overline X\) is indecomposable, all except one of the \(e_i\) has to be such that \(\overline{e_i}=0\). Therefore we may assume that \(\overline{e_1}\neq0\) and \(\overline{e_i}=0\) for \(i\in\left\{2,\ldots,n\right\}\). Note that \(\overline{\rho_i}\) is an epimorphism, since \(\rho_i\circ\iota_i=1_{X_i}\) so that \(\overline{\rho_i}\circ\overline{\iota_i}=\overline{\rho_i\circ\iota_i}=\overline{1_{X_i}}=1_{\overline{X_i}}\). Therefore \(\overline{\iota_i\circ\rho_i}=0\) means that \(\iota_i\in\mathcal J\). However, since \(\iota_i\in\mathcal J\) and \(\overline{\iota_i\circ\rho_i}=0\) this means that we also have \(\rho_i\in\mathcal J\). Then \(1_{\overline{X_i}}=0\) and so \(\overline{X_i}=0\). Looking at the endomorphism ring of \(\overline{X}\) we then have that \[\End_{\mathcal T/\mathcal J}(\overline{X})=\End_{\mathcal T/\mathcal J}(\overline{X_i})\] which is a local ring. 
\end{proof}
For more details we refer the reader to \cite{mi}, section 2 and 3.

Let \(\mathcal{A}\) be an abelian Hom-finite Krull-Schmidt \(k\)-category with finitely many isomorphism classes of indecomposable objects. We call a projective object \(P\) in \(\mathcal{A}\) a projective generator if for any object \(X\) in \(\mathcal{A}\) there is an epimorphism \( P^n\twoheadrightarrow X\) for some \(n\in\mathbb{N}\).

Our first aim is to establish that \(\mathcal{A}\) has a projective generator. We need this when we study abelian quotients \(\mathcal{A}=\mathcal{T}/\mathcal{J}\) in later sections.
The existence of the projective generator was shown for the case when all objects of \(\mathcal A\) has finite length by Deligne in \cite{deligne}. We do not know that we have finite length yet, but we do know that \(\mathcal A\) is a Krull-Schmidt category. Therefore we give a different proof. We will use the Harada-Sai lemma for the proof, so we need to recall the standard definition of length in a category.

\begin{definition}
An object \(X\) in an abelian category \(\mathcal{A}\) has finite length if there exists a finite chain of subobjects
\[0=X_0 \subsetneq X_1 \subsetneq \ldots \subsetneq X_{n-1}\subsetneq X_n=X\]
such that each quotient \(X_{i+1}/X_i\) is a simple object.
\end{definition}

We also define a different measure on the indecomposable objects of \(\mathcal A\). This will help us show that every object in \(A\) has finite length.
By \(\ind \mathcal{A}\) we denote the set isomorphism classes of indecomposable objects in \(\mathcal{A}\).

\begin{definition}Let \(X\) be an indecomposable object in \(\mathcal{A}\). We define \[\hat{l}(X)=\sum_{I\in\ind \mathcal A}\dim_k\Hom_{\mathcal{A}}(I,X).\]
\end{definition}

Since \(\mathcal{A}\) is \(\Hom\)-finite and there are finitely many isomorphism classes of indecomposables, \(\hat{l}(X)\) must be a finite number.

\begin{lemma}
Let \(X\) and \(Y\) be objects in \(\mathcal{A}\). If there exists a proper monomorphism \(i:X\rightarrow Y\) then \(\hat{l}(X)<\hat{l}(Y)\).
\end{lemma}
\begin{proof}
Assume that \(i:X\rightarrow Y\) is a proper monomorphism. For any \(I\), the induced morphism \( \Hom_\mathcal A(I,X)\rightarrow \Hom_\mathcal A(I,Y)\) is an inclusion. Therefore \(\dim\Hom_{\mathcal{A}}(I,X)\leq\dim\Hom_{\mathcal{A}}(I,Y)\) for all \(I\), and \(\hat{l}(X)\leq\hat{l}(Y)\).

The identity \(1_Y\) cannot factor through \(i\), as \(i\) is assumed not to split.
Therefore there is at least one indecomposable summand \(Y'\) of \(Y\) such that \(1_{Y'}:Y'\hookrightarrow Y\) does not factor through \(i:X\rightarrow Y\). Therefore \(\dim\Hom_{\mathcal A}(Y',X)<\dim\Hom_{\mathcal A}(Y',Y)\), and we must have \(\hat{l}(X)<\hat{l}(Y)\).
\end{proof}

\begin{theorem}
Let \(X\) be an object in an abelian Krull-Schmidt \(\Hom\)-finite \(k\)-category with finitely many indecomposable objects. Then \(X\) has finite length.
\end{theorem}

\begin{proof}
Consider a finite chain of subobjects of \(X\)
\[0=X_0 \subsetneq X_1 \subsetneq \ldots \subsetneq X_{n-1}\subsetneq X_n=X\]
where not all quotients \(X_{i+1}/X_i\) are necessarily simple objects. Choose a non-simple quotient \(X_{j+1}/X_j\). Let \(Z\) be a nonzero, proper subobject of \(X_{j+1}/X_j\). Consider the following commutative diagram with short exact rows.

\begin{center}
\begin{tikzpicture}
\node(A) at(0,0) {\(X_j\)};
\node(B) at (2,0) {\(X_{j+1}\)};
\node(C) at (4,0) {\(X_{j+1}/X_j\)};
\node(D) at (0,2){\(X_j\)};
\node(E) at (2,2) {\(Y\)};
\node(F) at (4,2) {\(Z\)};
\draw[->] (F)--(C);
\draw[->>](B)--(C);
\draw[right hook->](A)--(B);
\draw[right hook->](E)--(B);
\draw[right hook->](D)--(E);
\draw[->>](E)--(F);
\draw[double equal sign distance](A)--(D);
\end{tikzpicture}
\end{center}

The object \(Y\) is the pullback of \(Z\rightarrow X_{j+1}/X_j\) and \(X_{j+1}\rightarrow X_{j+1}/X_j\).
We see that \(Y\) is such that \(X_j\subseteq Y \subseteq X_{j+1}\). We now need to show the inclusions to be proper inclusions. 

If \(X_j=Y\), then \(Z\rightarrow X_{j+1}/X_j\) is an epimorphism, contradicting the choice of \(Z\). If \(Y=X_{j+1}\), the exact sequence in the top row of the diagram would force \(Z=0\), which again contradicts the choice of \(Z\). 

Hence we find a refined finite chain
\[0=X_0 \subsetneq X_1 \subsetneq \ldots \subsetneq X_j\subsetneq Y\subsetneq X_{j+1}\subsetneq\ldots  \subsetneq X_{n-1}\subsetneq X_n=X.\]
If the quotients of this chain are still not simple, the process can be repeated. However, since \(\hat l(X)\) is finite and \(\hat l(X_i)<\hat l(X_{i+1})\), we can only do a finite number of iterations of the process before reaching a chain where all the quotients are simple. Thus any object has a finite composition series and also finite length.
\end{proof}

Our aim now is to show that there are enough projectives in the category \(\mathcal{A}\). In order to achieve this the following lemma will be useful:

\begin{lemma}(Harada-Sai)

Consider a chain of length \(2^n\) of non-isomorphisms \(f_i\) between indecomposable objects of maximal length \(n\): \[X_1\stackrel{f_1}{\rightarrow} X_2\stackrel{f_2}{\rightarrow}\cdots X_{2^n-1}\stackrel{f_{2^n}}{\rightarrow}X_{2^n}.\] Then the composition \(f_{2^n} f_{2^n-1}\cdots f_2f_1\) is zero.
\end{lemma}

A proof of this lemma, which is also valid in abelian categories can for example be found in \cite{Rowen}.

For each indecomposable object \(X\) of \(\mathcal{A}\) we will show that there is a projective object with an epimorphism to \(X\) by iterating a certain process, which will build a tree with \(X\) as the root-node. Let \(N\) be the maximal length of any indecomposable object of \(\mathcal{A}\). Then by the Harada-Sai lemma, any chain of \(2^N\) non-isomorphisms in \(\mathcal{A}\) is zero.

\begin{theorem}
An abelian category \(\mathcal A\) with finitely many isomorphism classes of indecomposable objects has enough projectives.
\end{theorem}

\begin{proof}
Let \(X\) be an indecomposable, non-zero object of \(\mathcal{A}\) which is not projective. Let \(X^+\) be an object in \(\mathcal{A}\) with an epimorphism \(f^+\) to \(X\) which is not split. If \(X^+\) is projective, we are done. If \(X^+\) is not projective, decompose \(X^+\)  into a finite sum of indecomposable objects: \[X^+=\bigoplus_{i=1}^mX_i\] with morphisms \(f_i:X_i\rightarrow X\). Note that each morphism \(f_i\in\rad(X_i,X)\).

Now, consider each summand \(X_i\). If \(X_i\) is projective, we take no further action. Otherwise, we can again find an object \(X_i^+\) of \(\mathcal{A}\) with a non-split epimorphism \(f_i^+:X_i^+\rightarrow X_i\).
The objects \(X_i^+\) can be decomposed into finite sums of indecomposable objects again, and we iterate the process. 

This iteration process builds a directed tree, where each node is an indecomposable object of \(\mathcal{A}\), and each edge is a radical morphisms. We continue the iteration process until all leaf nodes of the tree are either projective indecomposable objects, or at a branch of length \(2^N\). The sum of the compositions of morphisms along all paths from the leaf nodes to the root nodes is an epimorphism by construction. Consider this morphism: 
\[\bigoplus_{\textnormal{leaf nodes}}X_{\mathrm{leaf}}\xrightarrow{g}X.\]
If there are no projective leaf nodes in the tree, g is a composition of \(2^N\) radical morphisms, and thus zero. Since it is an epimorphism, \(X=0\), which is a contradiction of the initial assumptions.

Let \(P\) be the sum of all projective leaf nodes occurring in the tree, we now know \(P\neq 0\). Consider the inclusion \(i\) from \(P\) to the sum of the leaf nodes. It is easy to see that \(g\circ i \) is an epimorphism from a projective object to \(X\).
\end{proof}


\section{The quotient functor is representable}
In the remaining sections we assume that \(\mathcal{T}/\mathcal{J}\) is an abelian category denoted by \(\mathcal A\). As established in the previous section \(\mathcal{A}\) has a projective generator when it is a finite category. For the rest of the article, neither \(\mathcal A\) nor \(\mathcal T\) are required to be finite. However, we require \(\mathcal A\) to have a projective generator \(P\).

In this section we define \(T\) as the minimal preimage of \(P\) in \(\mathcal{T}\). We first show that the functor \(\HomT-\) takes the ideal \(\mathcal J\) to zero in the module category \(\modf \End_\mathcal T(T)^\op\). Second, we show that the \(k\)-algebras \(\Gamma \coloneqq \End_{\mathcal{T}}(T)^{\op}\) and \(\Lambda\coloneqq\End_{\mathcal{A}}(P)^{\op}\) are isomorphic. Finally, the main result of the section is proved, namely that the quotient functor \(\pi\) is naturally isomorphic to \(\HomT-\).

\begin{definition}
Let \(\mathcal T\) be a triangulated category and \(\mathcal J\) an ideal in \(\mathcal T\) such that \(\mathcal A=\mathcal T / \mathcal J\) is an abelian category with a  basic projective generator \(P\). Let \(\pi\) be the quotient functor from \(\mathcal T\) to \(\mathcal A\), and assume \(\pi\) is cohomological. We define the minimal preimage of \(P\) in \(\mathcal T\) to be the basic object \(T\in\mathcal T\) such that:
\begin{itemize}
	\item \(\pi(T)=P\)
	\item for all indecomposable summands \(T'\) of \(T\) we have \(\pi(T')\neq0\).
\end{itemize}
\end{definition}

The following lemma will prove useful in the remaining sections.

\begin{lemma} \label{splitindec}
Let \(f:X\rightarrow Y\) be a morphism in \(\mathcal T\). Then
\begin{enumerate}
\item if \(Y\) is indecomposable and \(\pi(Y)\neq0\), then \(f\) is a split epimorphism if and only if \(\pi(f)\) is a split epimorphism.
\item if \(X\) is indecomposable and \(\pi(X)\neq0\), then \(f\) is a split monomorphism if and only if \(\pi(f)\) is a split monomorphism.
\end{enumerate}
\end{lemma}
\begin{proof}
We only prove the first statement, as the second is dual.

If \(f\) is a split epimorphism, then there exists a morphism \(g:Y\rightarrow X\) such that \(fg=1_Y\). But then \(\pi(f)\pi(g)=\pi(1_Y)=1_\pi(Y)\), so \(\pi(f)\) is a split epimorphism. 

If \(\pi(f)\) is a split epimorphism, there exists a morphism \(g':\pi(Y)\rightarrow \pi(X)\) such that \(\pi(f)g'=1_{\pi(Y)}\). Since \(\pi\) is a full functor, there exists a morphism \(g:Y\rightarrow X\) with \(\pi(g)=g'\). Since \(fg\in\End_{\mathcal T}(Y)\), which is a local ring, \(fg\) is either nilpotent or an isomorphism. As \(\pi(fg)=\pi(f)g'=1_{\pi(Y)}\), it clearly cannot be nilpotent. Hence \(fg\) is an isomorphism, and thus \(f\) is split epimorphism.
\end{proof}

For the rest of this section we fix \(P\) as the projective generator of \(\mathcal A\), and we fix \(T\) to be the minimal preimage of \(P\) in \(\mathcal T\). From now on we will denote a morphism \(\HomT f\) by \(\overline f\).

\begin{lemma}\label{idealhom}
Let \(f\in\mathcal J\). Then \(\Hom_{\mathcal T}(T, f)=\bar f=0\).
\end{lemma}

\begin{proof}
Assume that we have \(\bar f\neq0\), where \(f:X\rightarrow Y\). For at least one indecomposable \(T'\in\add T\), there exists at least one map \(g:T'\rightarrow X\) such that the composition \(T'\xrightarrow{g} X\xrightarrow f Y\) is non-zero. Since \(f\in \mathcal J\), we also have that \(fg \in \mathcal J\).
The morphism \(T'\xrightarrow{fg}Y\) can be completed to the triangle  \(Z\xrightarrow{h}T'\xrightarrow{fg}Y\rightarrow \Sigma Z\).

Since \(\pi\) is cohomological, we get the following exact sequence in \(\mathcal A\): \[\pi (Z)\xrightarrow{\pi(h)}\pi(T')\xrightarrow{\pi(fg)=0}\pi(Y).\] 
Since \(fg\in\mathcal J\) we have \(\pi(fg)=0\) in \(\mathcal A\), so \(\pi(h)\) is an epimorphism.
Since \(T'\) is an indecomposable summand of \(T\) we have that \(\pi(T')\) is an indecomposable projective in \(\mathcal A\), and hence \(\pi(h)\) is split epi. By lemma \ref{splitindec}, \(h\) is split epi, giving a morphism \(u\) such that \(1_{T'}=hu\).

From the distinguished triangle \(Z\stackrel{h}{\rightarrow}T'\stackrel{fg}{\rightarrow} Y\rightarrow \Sigma Z\) we see that the composition \(fgh=0\). By composing with \(u\), we get that \(0=fghu=fg\), which is a contradiction. Hence \(\bar f = 0\).
\end{proof}

\begin{lemma}
Let \(\Gamma=\End_{\mathcal T}(T)^{op}\) and \(\Lambda=\End_{\mathcal A}(P)^{op}\). 
Then \(\Gamma\) and \(\Lambda\) are isomorphic as \(k\)-algebras.
\end{lemma}

\begin{proof}
We know that \(\pi\) is a full and dense \(k\)-functor between \(\mathcal T\) and \(\mathcal A\). Hence it induces an algebra epimorphism 
\[\pi: \Gamma\rightarrow \End_{\mathcal A}(\pi(T))^{op}=\Lambda.\] 
It remains to show that \(\pi\) is a monomorphism as well.

Let \(f\in\Gamma=\End_{\mathcal T}(T)^{op}\) be such that \(\tilde\pi(f)=0\) in \(\mathcal A\). Then \(f\) is in the ideal \(\mathcal J\), and hence \(\HomT f=0\) by lemma \ref{idealhom}. This means that for any \(g\in\Gamma\) we must have \(gf=0\). In particular, \(f=1_Tf=0\). Hence \(\ker(\tilde\pi)=0\) and \(\Gamma \cong \Lambda\) as \(k\)-algebras.
\end{proof}

From \cite{mitchell} it is known that since \(\mathcal A\) has a projective generator \(P\), there is an equivalence of categories \(\mathcal A\cong \modf \Lambda=\modf \End_\mathcal A(P)^{\op}\). From the previous result we now know that \(\modf \Gamma \cong \modf \Lambda\). That is, we have two functors \(\pi\) and \(\HomT-\) such that \(\pi,\HomT{-}:\mathcal{T}\rightarrow \modf\Gamma\). Next we show that these two functors are naturally isomorphic.

\begin{theorem}
Let \(\pi:\mathcal T\rightarrow \mathcal A\) be a quotient functor from a triangulated category to an abelian category.
Let \(P\) be the projective generator of \(\mathcal A\), and let \(T\) be its minimal preimage in \(\mathcal T\). 

Then \(\pi\) is naturally isomorphic to \(\HomT -\).
\end{theorem}

\begin{proof}
By \cite{mitchell}, the equivalence of categories between \(\mathcal A\) and \(\modf \Lambda\) is given by \(\Hom_{\mathcal A}(P,-)\). 

Let \(X\) be an object in \(\mathcal T\). Consider the map \[\eta:\Hom_{\mathcal T}(T,X)\rightarrow\Hom_{\mathcal A}(P,\pi(X))\] induced by \(\pi\) (recall that \(P=\pi(T)\)). This is an epimorphism, since \(\pi\) is a full functor. Let \(g\in \HomT X\). If \(\pi(g)=0\), then \(g\in\mathcal J\), so \(\Hom_{\mathcal T}(T,g)=0\). Hence \(g=0\), and \(\eta\) is an isomorphism.

For any object \(X\) in \(\mathcal T\) we thus know that \[\pi(X)\cong\Hom_{\mathcal A}(P,\pi(X))\cong\Hom_{\mathcal C}(T,X).\] We will show that this is a natural transformation.

Let \(f:X\rightarrow Y\) be a morphism in \(\mathcal T\). Consider the following diagram:
\begin{center}
\begin{tikzpicture}
\node(A1) at (8,2) {\(\HomT X\)};
\node(D1) at (4.5,2){\(\Hom_{\mathcal A}(P,\pi(X))\)};
\node(C1) at (0,2){\(\pi(X)\)};
\node(A2) at (8,0){\(\HomT Y\)};
\node(D2) at (4.5,0){\(\Hom_{\mathcal A}(P,\pi(Y))\)};
\node(C2) at (0,0){\(\pi(Y)\)};
\draw[->] (C1)--node[anchor=west]{\(\pi(f)\)}(C2);
\draw[->](D1)--node[anchor=west]{\(\Hom_{\mathcal A}(P,\pi(f))\)}(D2);
\draw[->](A1)--node[anchor=west]{\(\HomT f \)}(A2);
\draw[->] (C1)--node[anchor=south]{\(\Hom_\mathcal{A}(P,-)\)}(D1);
\draw[->] (C2)--node[anchor=south]{\(\Hom_\mathcal{A}(P,-)\)}(D2);
\draw[->] (D1)--node[anchor=south]{\(\cong\)}(A1);
\draw[->] (D2)--node[anchor=south]{\(\cong\)}(A2);
\end{tikzpicture}
\end{center}
The first square commutes because \(\Hom_{\mathcal A}(P,\pi(f))=\Hom_\mathcal{A}(P,-)\circ\pi(f)\). The second square commutes by the functoriality of \(\pi\). Hence we have defined a natural transformation. Since the map for each object is an isomorphism, it is also a natural isomorphism.

\end{proof}


\section{When is \(\HomT{-}\) a quotient functor?}\label{FullDense}
In the previous section, we showed that the quotient functor is representable. This poses the question of when representable functors are quotient functors.

In this section we give our main result. It concerns technical conditions on an object \(T\) that are equivalent to \(\HomT{-}\) being a quotient functor (i.e. full and dense). 

We start by giving two useful lemmas. The first is well known, and the proof in \cite[ch.~II.~2]{ARS} extends our case.  For more details see e.g.\  
\cite{krause}.
Recall that \(\Gamma=\End_\mathcal T (T)\op\).
\begin{lemma}\label{addproj}
Let \(T\) be an arbitrary object in \(\mathcal T\).
\(\HomT -\) induces an equivalence \[\add T\cong \proj \Gamma\]
\end{lemma}
The second lemma is an extension of the first. Recall that we write \(\HomT f=\overline f\).

\begin{lemma}\label{fromTenough}
Let \(T\) be an arbitrary object in \(\mathcal T\).
Let \(T_0\in\add T\), let \(X, Y\in\mathcal T\) and let \(T_0\xrightarrow f X\), \(T_0\xrightarrow g Y\) and \(X\xrightarrow h Y\) be morphisms. 

If \(\overline g=\overline h\circ \overline f\), then \(g=hf\).
\end{lemma}

\begin{proof}
We have assumed that \(\mathcal T\) is Krull-Schmidt. Let \[T_0=\bigoplus_{j=0}^n T_0^j\] be the decomposition of \(T_0\) where the \(T_0^j\) are all indecomposable. We rewrite \(f=[f_0\cdots f_n]\) and \(g=[g_0\cdots g_n]\) with respect to this composition. Note that if for each \(j\) we have \(g_j=hf_j\), then \(g=hf\), so fix one \(j\).

Since \(\overline g=\overline h \overline f\), and \(\HomT -\) distributes over direct sums, we know that \(\overline g_j=\overline h\overline f_j\).
Since \(T_0^j\) is an indecomposable element in \(\add T\), \(T_0^j\) must be a summand of \(T\). Let \(i:T_0^j\rightarrow T\) and \(p:T\rightarrow T_0^j\) be the direct sum injection and projection respectively. We then construct the following commutative diagram:
\begin{center}
\begin{tikzpicture}
\node(TT0) at (0,2) {\(\HomT {T_0^j}\)};
\node(T0T0) at (0,0) {\(\Hom_\mathcal{T}(T_0^j, T_0^j)\)};
\node(TX) at (4,2) {\(\HomT {X}\)};
\node(T0X) at (4,0) {\(\Hom_\mathcal{T}(T_0^j, X)\)};
\node(TY) at (8,2) {\(\HomT {Y}\)};
\node(T0Y) at (8,0) {\(\Hom_\mathcal{T}(T_0^j, Y)\)};
\draw[->] (TT0) -- node[above]{\(\overline f_j\)}(TX);
\draw[->] (TX) -- node[above]{\(\overline h\)}(TY);
\draw[->] (T0T0) -- node[above=2pt]{\(\Hom_\mathcal{T}(T_0^j, {f_j})\)}(T0X);
\draw[->] (T0X) -- node[above=2pt]{\(\Hom_\mathcal{T}(T_0^j, h)\)}(T0Y);
\draw[->] (TT0) -- node[left]{\(\Hom_\mathcal{T}(i, {T_0})\)}(T0T0);
\draw[->] (TX) -- node[left]{\(\Hom_\mathcal{T}(i, {X})\)}(T0X);
\draw[->] (TY) -- node[left]{\(\Hom_\mathcal{T}(i, {Y})\)}(T0Y);
\end{tikzpicture}
\end{center}
We know that \(p\in\HomT{ T_0^j}\). By chasing \(p\) through the diagram, we get that 
\begin{align*}
g_j&=g_jp i 
\\&=[\Hom_\mathcal{T}(i, Y)\overline g_j](p) 
=[\Hom_\mathcal{T}(i, Y)\overline h\overline f_j](p)\\
&=[\Hom_\mathcal{T}(T_0^j,h)\Hom_\mathcal{T}(T_0^j, {f_i})\Hom_\mathcal{T}(i, {T_0})](p)\\
&=hf_jp i=hf_j
\end{align*}
\end{proof}

An object \(X\) is called \emph{\(T\)-supported} if \(\HomT X\neq 0\).

We are now ready to prove one of our main theorems.

\begin{theorem}\label{abfulldense}
\(\HomT{-}\) is a quotient functor, i.e.\ full and dense, if and only if the following two conditions are satisfied
\begin{description}
\item[a] For all right minimal morphisms \(T_1\rightarrow T_0\), where \(T_0, T_1\in\add T\), all triangles \(T_1\rightarrow T_0\rightarrow X\xrightarrow h \Sigma T_1\) satisfy \(\HomT h=0\).
\item[b] For all indecomposable, \(T\)-supported objects \(X\) there exists a triangle \(T_1\rightarrow T_0\rightarrow X\xrightarrow h \Sigma T_1\) with \(T_1,T_0\in \add T\) and \(\HomT h=0\).
\end{description}
\end{theorem}

\begin{proof}
Assume first that \textbf  a and \textbf b hold. We will first show that this means that \(\HomT{-}\) is dense, and then that it is full.

\textbf{a implies dense:}
Let \(X\) be an arbitrary object in \(\modf \Gamma\). We need to find an object \(Y\) in \(\mathcal T\) such that \(\HomT Y\cong X\).
We have the following minimal projective presentation of \(X\)
\[\HomT {T_1}\xrightarrow{\overline f} \HomT {T_0}\xrightarrow{g} X\rightarrow 0,\]
where \(T_1, T_0\in\add T\), by the equivalence \(\proj \Gamma\cong\add T\). 

Here, \(\overline f\) is the composition of the monomorphism \(\Ker g\rightarrow \HomT{T_0}\) and the projective cover \(\HomT{T_1}\rightarrow\Ker g\). 
The former is right minimal because it is a monomorphism. The latter is right minimal because it is a projective cover (see e.g.\ \cite[thm I.4.1]{ARS}). Hence  \(\overline f\) is right minimal.

The morphism \(f:T_1\rightarrow T_0\) in \(\mathcal T\) is right minimal by virtue of the equivalence \(\proj \Gamma\cong\add T\). We can complete \(f\) to an distinguished triangle 
\[T_1\xrightarrow f T_0\rightarrow Y\rightarrow \Sigma T_1.\]
Applying \(\HomT{-}\) and using \textbf a we get the following exact sequence:
\[\HomT {T_1}\xrightarrow{\overline f} \HomT {T_0}\rightarrow \HomT Y\rightarrow 0.\]
It follows from the uniqueness of cokernels that \(X\cong\HomT Y\), and thus \(\HomT -\) is dense.

\textbf{b implies full:} 
Let \(X\) and \(Y\) be two objects in \(\mathcal T\). Let \(f:\Hom_{\mathcal T}(T, X)\rightarrow\Hom_{\mathcal T}(T,Y)\) be an arbitrary morphism in \(\modf \Gamma\).
We need to find a morphism \(f':X\rightarrow Y\) in \(\mathcal T\) such that \(\overline{f'}=f\). Since the functor \(\HomT -\) distributes over direct sums, we assume without loss of generality that \(X\) and \(Y\) are indecomposable.

If \(X\) or \(Y\) is not \(T\)-supported, then obviously \(f=0\), so \(0:X\rightarrow Y\) maps to \(f\). In the following we therefore assume \(\HomT X\neq 0\) and \(\HomT Y\neq 0\)

Using property \textbf b, we define the following exact triangles: 

\begin{center}
\begin{tikzpicture}[scale=.90]
\node(T_1) at (0,1){\(T_1\)};
\node(T_0) at (2,1){\(T_0\)};
\node(T_1') at (0,0){\(T_1'\)};
\node(T_0') at (2,0){\(T_0'\)};
\node(X) at (4,1){\(X\)};
\node(Y) at (4,0){\(Y\)};
\node(ST_1) at (6,1){\(\Sigma T_1\)};
\node(ST'_1) at (6,0){\(\Sigma T'_1\)};
\draw[->](T_0)--node[anchor=south]{\(g\)}(X);
\draw[->](T_0')--node[anchor=south]{\(g'\)}(Y);
\draw[->](T_1)--(T_0);
\draw[->](T_1')--(T_0');
\draw[->](X)--(ST_1);
\draw[->](Y)--(ST'_1);
\end{tikzpicture}
\end{center}

By applying the functor \(\HomT -\) and using the comparison theorem for projective resolutions, we get the following diagram in \(\modf \Gamma\) with exact rows:

\begin{center}
\begin{tikzpicture}
\node(T_1) at (0,1.5){\(\HomT{T_1}\)};
\node(T_0) at (4,1.5){\(\Hom_{\mathcal T}(T,T_0)\)};
\node(T_1') at (0,0){\(\Hom_{\mathcal T}(T,T_1')\)};
\node(T_0') at (4,0){\(\Hom_{\mathcal T}(T,T_0')\)};
\node(X) at (8,1.5){\(\Hom_{\mathcal T}(T,X)\)};
\node(Y) at (8,0){\(\Hom_{\mathcal T}(T,Y)\)};
\node(0X) at (10,1.5){\(0\)};
\node(0Y) at (10,0){\(0\)};
\draw[->](X)--node[anchor=west]{\(f\)}(Y);
\draw[->](T_0)--node[anchor=south]{\(\overline g\)}(X);
\draw[->](T_0')--node[anchor=south]{\(\overline{g'}\)}(Y);
\draw[->](T_1)--(T_0);
\draw[->](T_1')--(T_0');
\draw[->](T_1)--(T_1');
\draw[->](T_0)--node[anchor=east]{\(\overline u\)}(T_0');
\draw[->](X)--(0X);
\draw[->](Y)--(0Y);
\end{tikzpicture}
\end{center}

By the equivalence between \(\add T\) and the projective objects in \(\modf \Gamma\), we can lift the left commutative square in the diagram back to \(\mathcal T\). The commutative square in \(\mathcal T\) can be completed to a morphism of triangles.

\begin{center}
\begin{tikzpicture}
\node(T_1) at (0,1.5){\(T_1\)};
\node(T_0) at (2,1.5){\(T_0\)};
\node(T_1') at (0,0){\(T_1'\)};
\node(T_0') at (2,0){\(T_0'\)};
\node(X) at (4,1.5){\(X\)};
\node(Y) at (4,0){\(Y\)};
\node(ST_1) at (6,1.5){\(\Sigma T_1\)};
\node(ST'_1) at (6,0){\(\Sigma T'_1\)};
\draw[->](T_0)--node[anchor=south]{\(g\)}(X);
\draw[->](T_0')--node[anchor=south]{\(g'\)}(Y);
\draw[->](T_1)--(T_0);
\draw[->](T_1')--(T_0');
\draw[->](T_1)--(T_1');
\draw[->](T_0)--node[anchor=east]{\(u\)}(T_0');
\draw[->](X)--(ST_1);
\draw[->](Y)--(ST'_1);
\draw[->](X)--node[anchor=west]{\(f'\)}(Y);
\draw[->](ST_1)--(ST'_1);
\end{tikzpicture}
\end{center}

Applying \(\HomT -\) once again, we get the following diagram:

\begin{center}
\begin{tikzpicture}
\node(T_1) at (0,1.5){\(\HomT{T_1}\)};
\node(T_0) at (4,1.5){\(\Hom_{\mathcal T}(T,T_0)\)};
\node(T_1') at (0,0){\(\Hom_{\mathcal T}(T,T_1')\)};
\node(T_0') at (4,0){\(\Hom_{\mathcal T}(T,T_0')\)};
\node(X) at (8,1.5){\(\Hom_{\mathcal T}(T,X)\)};
\node(Y) at (8,0){\(\Hom_{\mathcal T}(T,Y)\)};
\node(0X) at (10,1.5){\(0\)};
\node(0Y) at (10,0){\(0\)};
\draw[->](X)--node[anchor=west]{\(\overline{f'}\)}(Y);
\draw[->](T_0)--node[anchor=south]{\(\overline g\)}(X);
\draw[->](T_0')--node[anchor=south]{\(\overline{g'}\)}(Y);
\draw[->](T_1)--(T_0);
\draw[->](T_1')--(T_0');
\draw[->](T_1)--(T_1');
\draw[->](T_0)--node[anchor=east]{\(\overline u\)}(T_0');
\draw[->](X)--(0X);
\draw[->](Y)--(0Y);
\end{tikzpicture}
\end{center}

Since \[f \overline g=\overline{g'}\overline u=\overline{f'} \overline g\]
and \(\overline g\) is an epimorphism, it follows that \(f=\overline{f'}\), and we have shown \(\HomT -\) to be dense, thus finishing the first implication.

Next, assume that \(\HomT -\) is a full and dense functor. 

\textbf{Full and dense implies a:} Let \(f:T_1\rightarrow T_0\) be a right minimal morphism between objects in \(\add T\). Complete the morphism to the following triangle: 
\[T_1\xrightarrow f T_0\xrightarrow g X\xrightarrow h \Sigma T_1\] 
We want to show that \(\HomT h=0\). Use \(\HomT -\) on this triangle to obtain the diagram
\begin{center}
\begin{tikzpicture}[xscale=3.0, yscale=1.2]
\node(T_1) at (0,0){\(\HomT {T_1}\)};
\node(T_0) at (1,0){\(\HomT {T_0}\)};
\node(X) at (2,0){\(\HomT {X}\)};
\node(T_11) at (3,0){\(\HomT {\Sigma T_1}\)};
\node(Y) at (1.5,-1.5){\(\HomT {Y}\)};
\draw[->] (T_1) --node[anchor=south]{\(\overline f\)} (T_0);
\draw[->] (T_0) --node[anchor=south]{\(\overline g\)} (X);
\draw[->] (X) --node[anchor=south]{\(\overline h\)} (T_11);
\draw[->>] (T_0) --node[anchor=east]{\(\overline u\)} (Y);
\draw[right hook->] (Y) --node[anchor=west]{\(\overline v\)} (X);
\end{tikzpicture}
\end{center}
where \(\HomT Y=\im \bar g\). The image and the maps all have preimages in \(\mathcal T\), since \(\HomT -\) is full and dense.
We assume (without loss of generality) that all summands of \(Y\) are \(T\)-supported.
We want to show that \(Y\) is a direct summand of \(X\), and we start by showing that \(\HomT Y\) is a summand of \(\HomT X\).

By \(\overline u\circ\overline f=0\) and Lemma \ref{fromTenough}, we get that \(uf=0\).

Using \(\Hom_\mathcal{T} (-,Y)\) on the triangle, we get the exact sequence
\[\Hom_\mathcal{T} (X,Y)\xrightarrow{\Hom_\mathcal{T} (g,Y)}\Hom_\mathcal{T} (T_0,Y)\xrightarrow{\Hom_\mathcal{T} (f,Y)}\Hom_\mathcal{T} (T_1,Y)\]
Starting with \(u\in\Hom_\mathcal{T}(T_0,Y)\) we get that \(\Hom_\mathcal{T}(f,Y)(u)=uf=0\), so \(u\) must be in the image of \(\Hom_\mathcal{T}(g,Y)\). Thus there exists some \(w\in\Hom_\mathcal{T}(X,Y)\) with \(wg=u\).

Since \[\overline w\circ \overline v\circ \overline u=\overline w\circ\overline g=\overline u,\] and \(\overline u\) is an epimorphism, we must have \[\overline w\circ\overline v=1_{\HomT Y}.\]
Thus \(\HomT Y\) is a direct summand of \(\HomT X\).

In order to show that \(Y\) is a direct summand of \(X\), we first show that \(u\) is a left minimal morphism. Use the functor \(\Hom_\mathcal{T} (-,Y)\) to obtain \[\Hom_\mathcal{T} (Y,Y)\xrightarrow{\Hom_\mathcal{T} (u,Y)}\Hom_\mathcal{T} (T_0,Y)\] in \(\modf\End_\mathcal T Y\). By \cite[thm I.2.2]{ARS}, there exists a decomposition \[\Hom_\mathcal{T} (Y,Y)\cong \Hom_\mathcal{T} (Y_1,Y)\oplus \Hom_\mathcal{T} (Y_2,Y)\] such that \(u^\ast=\Hom_\mathcal{T} (u,Y)|_{\Hom_\mathcal{T} (Y_1,Y)}\) is a right minimal morphism and \\ \(\Hom_\mathcal{T} (u,Y)|_{\Hom_\mathcal{T} (Y_2,Y)}=0\). The preimages in \(\mathcal T\) exist by the equivalence between \(\add Y\) and \(\proj\End_\mathcal T Y\).

By the dual of lemma \ref{addproj}, there exists a preimage \(u_1\) of \(u^\ast\) such that \(u_1=u|_{Y_1}\), and \(u=\sm{u_1\\0}\).
Suppose that for a morphism \(x:Y_1\rightarrow Y_1\), we have \(xu_1=u_1\). Then \(\Hom_\mathcal{T}(u_1,Y)\Hom_\mathcal{T}(x,Y)=\Hom_\mathcal{T}(u_1,Y)\). We have that \(u^\ast=\Hom_\mathcal{T}(u_1,Y)\) is right minimal; thus \(\Hom_\mathcal{T}(x,Y)\) is an isomorphism. By (the dual of) lemma \ref{addproj}, \(x\) is also an isomorphism. Consequently, \(u_1\) is left minimal.

We know that \(\overline u={\sm{\overline u_1 \\0}}\).
Since \(\overline u\) is an epimorphism, \(Y_2\) cannot be \(T\)-supported. By choice of \(Y\) we have \(Y_2=0\). Therefore \(u=u_1\) is left minimal.

We have \(wvu=wg=u\). Thus \(wv\) is an isomorphism and \(Y\) is a direct summand of \(X\).

We rewrite the original triangle to
\begin{equation}\label{ShowR0}
T_1\xrightarrow f T_0\xrightarrow{\sm{u\\0}}Y\oplus R\xrightarrow{\sm{h_Y &h_R}}\Sigma T_1
\end{equation}
where \(h_Y=h|_Y\) and \(h_R=h|_R\). The next step is to show that \(R\) is a direct summand of \(\Sigma T_1\). Using \(\Hom_\mathcal T(-,R)\) on the triangle, we get the following exact sequence:
\[\Hom_\mathcal T(\Sigma T_1, \! R) \! \xrightarrow{\Hom_\mathcal T(\sm{h_Y &h_R}, \! R)} \! \Hom_\mathcal T(Y\oplus R,R) \! \xrightarrow{\Hom_\mathcal T\left(\sm{u\\0},R\right)} \! \Hom_\mathcal T(T_0,\! R)\]
The projection \(p_R:Y\oplus R\rightarrow R\) is contained in \(\Hom_\mathcal T(Y\oplus R,R)\). It is obviously in the kernel of \(\Hom_\mathcal T\left(\sm{u\\0},R\right)\), so it must be in the image of \(\Hom_\mathcal T(\sm{h_Y &h_R},R)\). Thus \(h_R\) is a split monomorphism, and we have \(\Sigma T_1=R\oplus S.\) The one-sided inverse of \(h_R\) we denote as \(z\).

Let \(y\in\End_\mathcal T(\Sigma T_1)\) be such that \((\Sigma f)\circ y=\Sigma f\). Since \(\Sigma\) is an autoequivalence we have \(f\circ \Sigma y=f\). Since \(f\) is right minimal,  \(\Sigma y\) is an isomorphism. But that means that \(y\) is an isomorphism as well. Thus \(\Sigma f\) is right minimal, and so is \(\overline{\Sigma f}\).

Now consider 
\begin{align*}
\overline{\Sigma f}&=\overline{\Sigma f}\circ 1_{\HomT{\Sigma T}}\\
&=\overline{\Sigma f}\circ\bm{\overline{h_R}\circ\overline{z}&0\\0&1_{\HomT S}}=\overline{\Sigma f}\circ\bm{0&0\\0&1_{\HomT S}}
\end{align*}
Consequently \(\sm{0&0\\0&1_{\HomT S}}\) is an isomorphism. Thus \(R\) is not \(T\)-supported. 

It follows that \(\overline h=0\), and \textbf a holds.

\textbf{Full and dense implies b:}
Let \(X\) be a \(T\)-supported indecomposable object in \(\mathcal T\). Consider its minimal projective presentation in \(\modf\Gamma\):
\[\HomT{T_1}\xrightarrow{\overline  f}\HomT{T_0}\xtwoheadrightarrow{\overline g}\HomT X\]
We know that \(\overline f\) is right minimal, and by lemma \ref{addproj}, so is \(f\).
Complete \(T_1\xrightarrow f T_0\) to a triangle in \(\mathcal T\)
\[T_1\xrightarrow f T_0\xrightarrow u Y\xrightarrow v \Sigma T_1,\]
where \(\overline v=0\) by condition \textbf a.

We use \(\HomT -\) on the triangle to obtain the following commutative diagram with exact rows:
\begin{center}
\begin{tikzpicture}[xscale=3.1, yscale=1.5]
\node(T_1) at (0,1){\(\HomT {T_1}\)};
\node(T_0) at (1,1){\(\HomT {T_0}\)};
\node(Y) at (2,1){\(\HomT {Y}\)};
\node(T_11) at (3,1){\(\HomT {\Sigma T_1}\)};
\node(T_1x) at (0,0){\(\HomT {T_1}\)};
\node(T_0x) at (1,0){\(\HomT {T_0}\)};
\node(X) at (2,0){\(\HomT {X}\)};
\node(0) at (3,0){0};

\draw[->] (T_1) --node[anchor=south]{\(\overline f\)} (T_0);
\draw[->] (T_0) --node[anchor=south]{\(\overline u\)} (Y);
\draw[->] (Y) --node[anchor=south]{\(\overline v=0\)} (T_11);
\draw[->] (T_1x) --node[anchor=south]{\(\overline f\)} (T_0x);
\draw[->] (T_0x) --node[anchor=south]{\(\overline g\)} (X);
\draw[->] (X) --node[anchor=south]{\(0\)} (0);

\draw[double equal sign distance] (T_1) -- (T_1x);
\draw[double equal sign distance] (T_0) -- (T_0x);
\end{tikzpicture}
\end{center}

By uniqueness of cokernels  we must have \(\HomT X\cong\HomT Y\). Thus there are maps \(x:X\rightarrow Y\) and \(y:Y\rightarrow X\) such that 
\begin{align*}
\overline x\circ\overline y=1_{\HomT Y} \textnormal{\, and \,}
\overline y\circ\overline x=1_{\HomT X}.
\end{align*}
Once again, the preimages of the isomorphisms exist due to \(\HomT -\) being full.

We have that \(yx\in \End_\mathcal T (X)\), which is a local ring since \(X\) was assumed to be indecomposable. Since \(\overline{yx}\) is an isomorphism, \(yx\) clearly cannot be nilpotent, so it must be an isomorphism. Thus \(X\) is a direct summand of \(Y\), and we write \(Y=X\oplus R\).

We are now in the same situation as diagram (\ref{ShowR0}), and we can use the same argument to show that \(\Sigma T_1=R\oplus S\). By \cite[Lemma 1.2.4]{neeman}, 
\[T_1\xrightarrow f T_0\rightarrow X\rightarrow \Sigma T_1\]
is a distinguished triangle. It fulfills the requirements of \textbf b.
\end{proof}


\begin{example} We revisit the last example presented in \cite{k-zhu}. This is an example of an abelian quotient of a triangulated category with no cluster-tilting subcategories, hence not covered by the theory developed in \cite{k-zhu}.  Let \(A=kQ/I\) be the self-injective algebra given by the quiver
\begin{center}
\begin{tikzpicture}
\node (Q) at (-1.5,0){\(Q:\)};
\node (a) at (-1,0){\(a\)};
\node (b) at (1,0){\(b\)};
\draw[->] ([yshift =  .4ex]a.east) -- node[above]{\(\alpha\)} ([yshift =  .4ex]b.west);
\draw[->] ([yshift =  -.4ex]b.west) -- node[below]{\(\beta\)} ([yshift = -.4ex]a.east);
\end{tikzpicture}
\end{center}
and the relations \(\alpha \beta\alpha, \  \beta\alpha \beta \).

\newcommand{\Maba}{\begin{smallmatrix} a \\ b \\ a \end{smallmatrix}}
\newcommand{\Mbab}{\begin{smallmatrix} b \\ a \\ b \end{smallmatrix}}
\newcommand{\Mab}{\begin{smallmatrix} a \\ b \end{smallmatrix}}
\newcommand{\Mba}{\begin{smallmatrix} b \\ a \end{smallmatrix}}

The AR-quiver of \(\modf A\) is

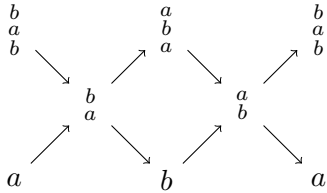
\begin{center}
\begin{tikzpicture}
\node(bab) at(-1,0){\(\Mbab\)};
\node(aba) at (1,0){\(\Maba\)};
\node(bab2) at (3,0){\(\Mbab\)};
\node(ba) at (0,-1){\(\Mba\)};
\node(ab) at (2,-1){\(\Mab\)};
\node(a) at(-1,-2){\(a\)};
\node(b) at (1,-2){\(b\)};
\node(a2) at(3,-2){\(a\)};
\draw[->](bab)--(ba);
\draw[->](ba)--(aba);
\draw[->](aba)--(ab);
\draw[->](ab)--(bab2);
\draw[->](a)--(ba);
\draw[->](ba)--(b);
\draw[->](b)--(ab);
\draw[->](ab)--(a2);
\end{tikzpicture}
\end{center}
where the first and the last columns are identified. The stable module category \(\underline{\modf} A\) is triangulated with suspension functor \(\Omega^{-1}\), the cozysygy. Its AR-quiver is:

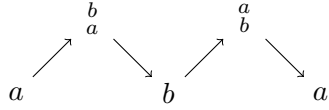
\begin{center}
\begin{tikzpicture}
\node(ba) at (0,-1){\(\Mba\)};
\node(ab) at (2,-1){\(\Mab\)};
\node(a) at(-1,-2){\(a\)};
\node(b) at (1,-2){\(b\)};
\node(a2) at(3,-2){\(a\)};
\draw[->](a)--(ba);
\draw[->](ba)--(b);
\draw[->](b)--(ab);
\draw[->](ab)--(a2);
\end{tikzpicture}
\end{center}

As explained in detail in \cite{k-zhu} this triangulated category does not have any cluster-tilting subcategories. An abelian quotient can be formed, by factoring out \(\add(a)\). This abelian category has the following AR-quiver:

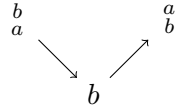
\begin{center}
\begin{tikzpicture}
\node(ba) at (0,-1){\(\Mba\)};
\node(ab) at (2,-1){\(\Mab\)};
\node(b) at (1,-2){\(b\)};
\draw[->](ba)--(b);
\draw[->](b)--(ab);
\end{tikzpicture}
\end{center}

The projective generator of this category is \(\Mba\oplus b\). The preimage of the projective generator is \(\Mba\oplus b\) considered as an object in \(\underline{\modf} A\). The functor \(\Hom_{\underline{\modf} A}(\Mba\oplus b,- )\) gives rise to the an abelian category with the same AR-quiver as \(\underline{\modf} A/\add(a)\). 

There is only one right minimal morphism between indecomposable objects of \(\add (\Mba \oplus b)\), namely \(\Mba\rightarrow b\). The triangle 
\[\Mba\rightarrow b\rightarrow \Mab\rightarrow\Omega^{-1}\Mba\]
shows that condition \textbf a is fulfilled.

The triangle also shows that condition \textbf b is fulfilled for \(\Mab\). For \(\Mba\) and \(b\), condition \textbf b is fulfilled by the completion of the identity morphism to a triangle. Hence theorem \ref{abfulldense} implies that \(\Hom_{\underline{\modf} A}(\Mba\oplus b,-)\) is full and dense.
\end{example}

\begin{example}
We revisit the class of examples mentioned in the introduction. We will study \(\mathcal{D}^b(kQ)/\Sigma\) where \(kQ\) is a Dynkin diagram. Specifically we consider the quiver \(A_3\) with orientation \(1\rightarrow 2\rightarrow 3\). The AR-quiver of \(\modf kA_3\) is:

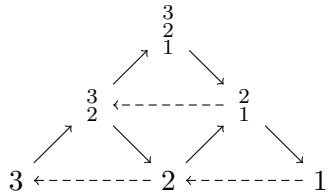
\newcommand{\Mtretoen}{\begin{smallmatrix} 3 \\ 2 \\ 1 \end{smallmatrix}}
\newcommand{\Mtreto}{\begin{smallmatrix} 3 \\ 2 \end{smallmatrix}}
\newcommand{\Mtoen}{\begin{smallmatrix} 2 \\ 1 \end{smallmatrix}}

\begin{center}
\begin{tikzpicture}
\node(tre) at (-1,0){\(3\)};
\node(to) at(1,0){\(2\)};
\node(en) at (3,0){\(1\)};
\node(treto) at (0,1){\(\Mtreto\)};
\node(toen) at (2,1){\(\Mtoen\)};
\node(tretoen) at (1,2){\(\Mtretoen\)};
\draw[->](tre)--(treto);
\draw[->](treto)--(to);
\draw[->](to)--(toen);
\draw[->](treto)--(tretoen);
\draw[->](tretoen)--(toen);
\draw[->](toen)--(en);
\draw[->,densely dashed](to)--(tre);
\draw[->,densely dashed](en)--(to);
\draw[->,densely dashed](toen)--(treto);
\end{tikzpicture}
\end{center}
The AR-quiver of the triangulated category \(\mathcal T=\mathcal{D}^b(kA_3)/\Sigma\) is:
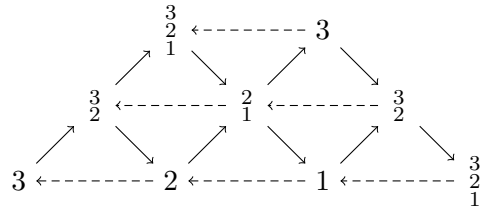
\begin{center}
\begin{tikzpicture}
\node(tre) at (-1,0){\(3\)};
\node(to) at(1,0){\(2\)};
\node(en) at (3,0){\(1\)};
\node(treto) at (0,1){\(\Mtreto\)};
\node(toen) at (2,1){\(\Mtoen\)};
\node(tretoen) at (1,2){\(\Mtretoen\)};
\node(tre2) at (3,2){\(3\)};
\node(treto2) at (4,1){\(\Mtreto\)};
\node(tretoen2) at(5,0){\(\Mtretoen\)};
\draw[->](tre)--(treto);
\draw[->](treto)--(to);
\draw[->](to)--(toen);
\draw[->](treto)--(tretoen);
\draw[->](tretoen)--(toen);
\draw[->](toen)--(en);
\draw[->](toen)--(tre2);
\draw[->](en)--(treto2);
\draw[->](tre2)--(treto2);
\draw[->](treto2)--(tretoen2);
\draw[->,densely dashed](to)--(tre);
\draw[->,densely dashed](en)--(to);
\draw[->,densely dashed](toen)--(treto);
\draw[->,densely dashed](tre2)--(tretoen);
\draw[->,densely dashed](treto2)--(toen);
\draw[->,densely dashed](tretoen2)--(en);
\end{tikzpicture}
\end{center}
where we include some objects twice to indicate which objects are identified. 

This category does not have any cluster-tilting subcategories, so it is not possible to attain an abelian quotient by the method used in \cite{k-zhu}. However it is in some sense already close to being an abelian category; the difference is just two irreducible maps! We let \(T=\Mtretoen\oplus\Mtreto\oplus 3\). Applying the functor \(\Hom_{\mathcal{T}}(\Mtretoen\oplus\Mtreto\oplus 3,-)\) we return to an abelian category equivalent to the module category. The functor is easily seen to be full and dense directly.

There are two right minimal morphisms between indecomposable summands of \(T\), namely \(3\rightarrow \Mtreto\) and \(\Mtreto\rightarrow \Mtretoen\). Thus the triangles \[3\rightarrow\Mtreto\rightarrow 2\rightarrow3\] and \[\Mtreto \rightarrow \Mtretoen \rightarrow 1\rightarrow\Mtreto\] show that part \textbf{a} of theorem \ref{abfulldense} is satisfied. 

The objects \(3,\Mtreto\) and \(\Mtretoen\) are in \(\add T\), so the completions of the identity maps on these objects fulfills condition \textbf b of theorem \ref{abfulldense}.
The above triangles fulfill condition \textbf b for the objects \(1\) and \(2\).
 The only object that remains is \(\Mtoen\), and in this case the triangle \[3\rightarrow \Mtretoen \rightarrow \Mtoen\rightarrow 3\] satisfies condition \textbf{b}. Hence by theorem \ref{abfulldense} the functor \(\Hom_{\mathcal{T}}(\Mtretoen\oplus\Mtreto\oplus 3,-)\) is full and dense. 
\end{example}


\section{AR-structure in the abelian quotient}
In this section we show that the AR-structure of \(\mathcal T\) is preserved as much as one can hope for in the abelian quotient. Let 
\[\Delta: \tau X \stackrel{f}{\longrightarrow} Y \stackrel{g}{\longrightarrow} X\stackrel{h}{\longrightarrow} \Sigma \tau X\] 
 be an AR-triangle in \(\mathcal T\). Assume that none of the objects in \(\Delta\) are sent to zero, that \(\tau X\) is not sent to an injective and that \(X\) is not sent to a projective. Then we show that \(\Delta\) is sent to an AR-sequence in \(\mathcal{A}\). 

Before proceeding we need to define two new notions. We call a category \(\mathcal{T}\) locally finite if for each indecomposable object \(X\) of \(\mathcal{T}\) there are only finitely many isomorphism classes of indecomposable objects \(Y\) such that \(\Hom_{\mathcal{T}}(X,Y)\neq0\). A Serre functor \(\mathbb{S}\) is an autoequivalence \(\mathbb{S}:\mathcal{T}\rightarrow\mathcal{T}\) such that for each object \(X\) in \(\mathcal{T}\) there is an isomorphism \[D\Hom_\mathcal{T}(X,-)\cong\Hom_\mathcal{T}(-,\mathbb{S}X)\] 
where \(D\) is the duality \(\Hom_k(-,k)\).

The following theorem is due to \cite{rvdb}.

\begin{theorem}
Let \(\mathcal{T}\) be a Hom-finite Krull-Schmidt \(k\)-category. Then \(\mathcal{T}\) has AR-triangles if and only if \(\mathcal{T}\) has a Serre functor \(\mathbb{S}\).
\end{theorem}

Many triangulated categories have a Serre functor. For example Amiot showed in \cite{amiot} that any locally finite Krull-Schmidt triangulated \(k\)-category has a Serre functor.

We assume in the following that \(\mathcal T\) has AR-triangles, and we proceed to study the AR-structure in the abelian quotient category.
As before we let \(T\in\mathcal T\) be an object such that \(\HomT -:\mathcal T\rightarrow \modf \Gamma\) is full and dense, where \(\Gamma=\End(T)^{op}\).

\begin{lemma}\label{proj0}
Let \[\Delta: \tau X \stackrel{f}{\longrightarrow} Y \stackrel{g}{\longrightarrow} X\stackrel{h}{\longrightarrow} \Sigma \tau X\] be an AR-triangle in \(\mathcal T\). Then  \(\overline h=0\) if and only if \(X\notin\add T\)
\end{lemma}

\begin{proof}
Assume first that \(\overline h=0\). Then \(\overline g\) is an epimorphism. If \(X\in\add T\), Then \(\HomT X\) is projective, and the epimorphism \(\overline g\) is split. By lemma \ref{splitindec}, \(g\) is also a split epimorphism.
However this leads to \(h=0\) which is a contradiction to \(\Delta\) being an AR-triangle, hence \(X\notin\add T\).

Now assume that \(X\notin\add T\). Then \(\HomT X\) is not projective. If \(X\) is not \(T\)-supported  then clearly \(\overline h=0\). If on the other hand \(X\) is \(T\)-supported then there exists a non-split epimorphism  \[\HomT A\stackrel{\overline u}{\longrightarrow}\HomT X,\] giving rise to a morphism \( A\stackrel{u}{\rightarrow} X\) which is not a split epimorphism. Since \(g\) is an almost split morphism, there exists a morphism \(v:A\rightarrow  Y\) such that \(u=g\circ v\). We have \(\overline u=\overline g\circ \overline v\), where \(\overline u\) is an epimorphism. Hence \(\overline g\) is an epimorphism and so \(\overline h=0\).
\qquad

\end{proof}

\begin{lemma}\label{inj0}
If \(\Sigma^{-1} X\in\add T\) then \(\HomT{\tau X}\) is injective and nonzero.
\end{lemma}
\begin{proof}
We have 
\[\HomT{\tau X}\cong\HomT{\mathbb S \Sigma^{-1} X}\cong D\Hom_\mathcal T(\Sigma^{-1} X,T).\]
Hence if \(\Sigma^{-1} X\in\add T\) then \(\Hom_\mathcal T(\Sigma^{-1} X,T)\in\proj \Gamma\op\). It follows that \(D\Hom_\mathcal T(\Sigma^{-1} X,T)\) is injective in \(\modf \Gamma\).
\end{proof}

\begin{lemma}\label{AR-triangle}
Let \(\Delta\) be the AR-triangle 
\[\Delta: \tau X \stackrel{f}{\longrightarrow} Y \stackrel{g}{\longrightarrow} X\stackrel{h}{\longrightarrow} \Sigma \tau X.\]
Assume that \(X\) and \(\tau X\) are both \(T\)-supported, with \(\Sigma^{-1} X\notin \add T\), and \(X\notin \add T\). Then the functor \(\HomT -\) takes the AR-triangle \(\Delta\) to the following AR-sequence in \(\modf \Gamma\):
\begin{equation}\label{ses}
0\rightarrow \HomT{\tau X} \stackrel{\overline f}{\longrightarrow} \HomT Y \stackrel{\overline g}{\longrightarrow}\HomT X\rightarrow 0
\end{equation}
\end{lemma}

\begin{proof}
From lemma \ref{proj0} it is clear that \(\HomT -\) takes \(\Delta\) to the short exact sequence (\ref{ses}). Note that \(\HomT{\tau X}\) and \(\HomT X\) are indecomposable since \(X\) and \(\tau X\) are indecomposable. Since \(\HomT{\tau X}\) is indecomposable it its enough to show that \(\overline g\) is right almost split \cite[thm V.1.14]{ARS}.

If \(\overline g\) is a split epimorphism then by lemma \ref{splitindec} we have that \(g\) is a split epimorphism, which contradicts the fact that \(\Delta\) is an AR-triangle. Hence \(\overline g\) is not a split epimorphism.

Assume that \( u: W\rightarrow X\) is a morphism such that \(\overline u\) is not a split epimorphism. By Lemma \ref{splitindec}, \(u\) is not a split epimorphism.

Since \(u\) is not a split epimorphism and \(g\) is right almost split, there is a morphism \(v:A\rightarrow Y\) such that \(u=g\circ v\). Applying \(\HomT -\) to this we obtain exactly what we want, which is a morphism \(\overline v:\HomT A\rightarrow\HomT Y\) such that \(\overline u=\overline g\circ\overline v\).
\end{proof}

\section{Cluster-tilting objects and the 2-Calabi-Yau case}
In this section we will work under the additional assumption that \(\mathcal T\) is a 2-Calabi-Yau category.  
We give two notable results. 
First we show for which objects \(T\) applying the functor \(\HomT{-}\) coincides with the cluster-tilting case studied in \cite{k-zhu}. 
Then we apply this result, to show that in many finite categories the only possible way to obtain an abelian quotient is with the previously known method from \cite{k-zhu}. 

We start by defining a cluster-tilting object.
\begin{definition}\label{cto}
An object \(T\) in a triangulated category \(\mathcal T\) is called a \emph{cluster-tilting object} if \[\add(T)=\{X|\HomT {\Sigma X}=0\}=\{X|\Hom_\mathcal T(X,\Sigma T)=0\}.\]
\end{definition}

Cluster-tilting objects turn out to be very closely related to the objects where \(T\) is such that \(\HomT{-}\) is a full and dense functor. In section \ref{FullDense} we showed that \(\HomT -\) is full and dense if and only if condition \textbf a and \textbf b were satisfied. We consider the following, stronger, version of condition \textbf b:
\begin{description}
\item[b*] For all indecomposable objects \(X\) there exists a triangle \[T_1\rightarrow T_0\rightarrow X\xrightarrow h \Sigma T_1\] with \(T_1,T_0\in \add T\) and \(\HomT h=0\).
\end{description}
The difference from \textbf b is that we require existence of such a triangle not only for objects \(X\) that are \(T\)-supported, but for all objects.

\begin{theorem}\label{abc}
An object \(T\) in \(\mathcal T\) is a cluster-tilting object if and only if \textbf a from Theorem \ref{abfulldense} and \textbf{b*}  are satisfied and furthermore: 
\begin{description}
\item[c] if \(T'\) is an indecomposable summand of \(T\), then \(\Sigma T'\notin\add T\)
\end{description}
\end{theorem}
\begin{proof}
Suppose \textbf a, \textbf{b*} and \textbf c holds, we need to show that \(T\) is cluster-tilting. 

Assume \(T'\in\add T\) is indecomposable. We need to show that \(\Sigma T'\) is not \(T\)-supported, i.e \(\HomT{\Sigma T'}=0\). By \textbf{b*} a distinguished triangle \[T_1\rightarrow T_0\rightarrow \Sigma T'\xrightarrow f \Sigma T_1\] exists,
 where \(T_1, T_0\in\add T\) and \(\overline f=0\). We have that \(f\neq 0\) by \textbf c.
Since \(T'\) is indecomposable, \(\Sigma^{-1} f:T'\rightarrow T_1\) is right minimal. By rotating the above triangle, we get the distinguished triangle \[T'\xrightarrow {\Sigma^{-1} f} T_1\rightarrow T_0\xrightarrow{g} \Sigma T'\]
 where \(\overline g=0\) by \textbf a. This means that the following sequence is exact: \[\HomT {T_0}\xrightarrow{\overline g=0}\HomT {\Sigma T'}\xrightarrow{\overline f=0}\HomT{ \Sigma T_1}.\]
Consequently \(\HomT{\Sigma T'}=0\).

Conversely, suppose \(X\) is such that \(\Hom_\mathcal T(T,\Sigma X)=0\). We need to show that \(X\) is in \(\add T\). By \textbf{b*} there exists \(T_0,T_1\in\add T\) such that following triangle is distinguished:
\[T_1\rightarrow T_0\xrightarrow{0} \Sigma X\rightarrow \Sigma T_1.\]
The zero follows from the assumption on \(X\).

Using the axioms for triangulated categories we see that the following triangle is also distinguished:
\[X\rightarrow T_1\rightarrow T_0\xrightarrow{0} \Sigma X.\]
By \cite{neeman}, this is a split triangle; thus \(T_1\cong X\oplus T_0\) and \(X\in \add T\).

We now have shown that \(\add T=\{X|\Hom_\mathcal T(T,\Sigma X)=0\}\); the other equality in definition \ref{cto} can be shown similarly. Thus \(T\) must be cluster-tilting.
Suppose now instead that \(T\) is a cluster-tilting object; we show \textbf a, \textbf{b*} and \textbf c in order.

Let \(T_1\rightarrow T_0\) be a right minimal morphism between objects in \(\add T\), and complete this morphism to a triangle:
\[T_1\rightarrow T_0\rightarrow Y\xrightarrow g \Sigma T_1.\]
Since \(\HomT{\Sigma T_1}=0\), we must also have \(\overline g=0\), and \textbf a holds.

Note that \(\add T\) is a functorially finite subcategory; this is well known and not dependent on \(T\) being a cluster-tilting object. In particular \(\add T\) is contravariantly finite.
We now follow \cite[thm 3.2]{k-zhu} to show that condition \textbf{b*} holds. 

Let \(X\) be an arbitrary object of \(\mathcal T\). Let \(f:T_0\rightarrow X\) be a right \(\add T\)-approximation of \(X\). We complete this to the triangle 
\[Y\rightarrow T_0\xrightarrow f X\xrightarrow g \Sigma Y.\]
Applying \(\HomT -\) we get the long exact sequence
\[\cdots\rightarrow\! \HomT{T_0}\!\xrightarrow{\overline f}\!\HomT X
\!\xrightarrow{\overline g}\!\HomT{\Sigma Y}\!\rightarrow\!\HomT{\Sigma T_0}\!\rightarrow\cdots
\]
We have \(\HomT{\Sigma Y}=0\), since \(\overline f\) is surjective and \(\HomT{\Sigma T_0}=0\). Consequently, \(Y\in \add T\), and condition \textbf{b*} holds.

To show condition \textbf c, assume that \(T'\) is an indecomposable summand of \(T\) with \(\Sigma T'\in\add T\). As \(\HomT{\Sigma T}=0\), we must have \(\HomT{\Sigma T'}=0\). Since \(\Sigma T'\in\add T\), this means that \(\Sigma T'=0\). Hence \(T'=0\).
\end{proof}

A triangulated category \(\mathcal{T}\) with a Serre functor \(\mathbb{S}\) is said to be \(d\)-Calabi-Yau if \(\mathbb{S}=\Sigma^d\), and \(d\) is the smallest positive integer for which this holds. In particular 2-Calabi-Yau categories have been studied quite extensively. Examples of such categories include the classical cluster categories, \(D^b(H)/\tau^-\Sigma\) where \(H\) is an hereditary algebra \cite{bmrrt}. 

We will show that in finite 2-Calabi-Yau categories if \(\HomT -\) is full and dense, then in most cases \(T\) must be a cluster-tilting object. To do this we first need to study the structure of finite 2-Calabi-Yau categories.

\begin{lemma}
Let \(\mathcal T\) be a connected 2-Calabi-Yau category with finitely many isomorphism classes of indecomposable objects. Let \(T\in\Ob \mathcal T\) be a non-zero object such that condition \textbf b is satisfied.
Then \[\ind \mathcal T=\{X\in\ind\mathcal T|\HomT X\neq 0\}\cup\{X\in\ind\mathcal T|X\in\add \Sigma T\}.\]
\end{lemma}
\begin{proof}

Let \(\mathcal D=\{X\in\ind\mathcal T|\HomT X\neq 0\}\cup\{X\in\ind\mathcal T|X\in\add \Sigma T\}\).
Note that \(\mathcal D\) is non-empty, as all indecomposable summands of \(T\) are in \(\mathcal D\).

We will show that for any object \(X\in\mathcal D\) we can find a triangle \[T_1\rightarrow T_0\rightarrow X\xrightarrow{h} \Sigma T_1\] such that \(T_1, T_0\in \add T\) and \(\HomT h=0\). If \(X\) is \(T\)-supported, this follows immediately from condition \textbf b. If not, then \(X\in \add\Sigma T\), so \(\Sigma^{-1}X\in\add T\). The following split triangle fulfills the conditions.
\[\Sigma^{-1}X \rightarrow 0\rightarrow X\rightarrow X\]

Assume that \(X\in \mathcal D\) and \(Y\in\ind \mathcal T\). Let \(f:X\rightarrow Y\) be a non-zero morphism. We will show that \(Y\in\mathcal D\). By the above we can form the following diagram:
\begin{center}
\begin{tikzpicture}[xscale=1.5]
\node(T1) at (0,1) {\(T_1\)};
\node(T0) at (1,1) {\(T_0\)};
\node(X)  at (2,1) {\(X\)};
\node(ST) at (3,1) {\(\Sigma T_1\)};
\node(Y) at (2,0) {\(Y\)};
\draw[->] (T1) -- node[anchor=south]{}(T0);
\draw[->] (T0) -- node[anchor=south]{\(g\)}(X);
\draw[->] (X) -- node[anchor=east]{\(f\)}(Y);
\draw[->] (X) -- node[anchor=south]{\(h\)}(ST);
\end{tikzpicture}
\end{center}
If \(gf\neq0\), then \(Y\) is \(T\)-supported. Hence \(Y\in\mathcal D\), and we are done.

If \(gf=0\), then by the weak cokernel property of triangulated categories, there exists a morphism \(f':\Sigma T_1\rightarrow Y\) such that \(f'h=f\). Hence \(\Sigma^ {-1}Y\) is \(T\)-supported. By condition \textbf b we can form a distinguished triangle \[T_1'\rightarrow T_0'\rightarrow \Sigma^{-1}Y\rightarrow \Sigma T_1',\]
with \(T'_1, T_0\in\add T\). Hence the triangle
\[\Sigma T_1'\rightarrow\Sigma T_0'\rightarrow Y\xrightarrow{h'} \Sigma^2 T_1'\]
is distinguished.
If \(h'=0\), then the triangle splits, and \(Y\) is a summand of \(\Sigma T_0'\). Hence \(Y\in\add \Sigma T\).

If \(h'\neq 0\), then
\[0\neq\Hom_\mathcal T(Y,\Sigma^2 T)=\Hom_\mathcal T(Y,\mathbb S T)\cong D\Hom_\mathcal T(T,Y).\]
Hence \(Y\) is \(T\)-supported.

In \cite{amiot}, the author shows that any connected triangulated category with finitely many indecomposables has an AR-quiver of the form \(\mathbb Z\Delta/G\), where \(\Delta\) is a Dynkin diagram and \(G\) is a group of weakly admissible automorphisms of \(\mathbb Z\Delta\). By corollary 6.3.3 in \cite{amiot} \( \mathcal T\) is an orbit category. Since, by the above, any \(\tau\)-orbit of the AR-quiver of \(\mathcal T\) contains an element of \(\mathcal D\), we see that \(\ind\mathcal T=\mathcal D\).
\end{proof}
A consequence of this lemma is that for the connected 2-Calabi-Yau case \textbf{b*} is implied by \textbf b. We are now ready to prove the final theorem. Recall that an object \(X\) in a triangulated category \(\mathcal T\) such that \(\End_{\mathcal T}(X)^{\op}\cong k\) is called a Schurian object.

\begin{theorem}
Let \(\mathcal T\) be a 2-CY connected triangulated category with finitely many isomorphism classes of indecomposable objects. If \(T\) is an object in \(\mathcal T\) such that \(\HomT -:\mathcal T\rightarrow \modf\Gamma\) is full and dense, then \(T\) is either Schurian or \(T\) a 2-cluster-tilting object in \(\mathcal T\).
\end{theorem}

\begin{proof}
Condition \textbf a is satisfied, since \(\HomT -\) is full and dense. By the above, so is condition \textbf{b*}. If \(T\) satisfies \textbf{c}, then by lemma \ref{abc} \(T\) is a cluster-tilting object.

Assume that \(T\) does not satisfy condition \textbf c. We will show that \(\modf \Gamma=\modf k\). 

Let \(T'\) be an indecomposable summand of \(T\) such that \(T',\Sigma T'\in\add T\). Since \(\mathcal T\) is a 2-CY triangulated category, we have \(\Sigma T'\cong\tau T'\). Therefore we have the following AR-triangle
\[\Delta:\tau T' \xrightarrow{f}E\xrightarrow g T'\xrightarrow h \Sigma\tau T'.\]
Applying \(\HomT -\) to the above AR-triangle yields the following long exact sequence: 
\[ \ldots\rightarrow \Hom_{\mathcal T}(T,\tau T')\xrightarrow{\overline{f}}\Hom_{\mathcal T}(T,E)\xrightarrow{\overline g}\Hom_{\mathcal T}(T,T')\rightarrow\ldots\] where by the proof of lemma \ref{AR-triangle} the map \(\overline g\) is right almost split. 

Since \(T'\) in \(\add T\), we have that \(\HomT{T'}\) is projective. Hence there exists a right almost split monomorphism \(r:\rad_{\Gamma}\Hom(T,T')\rightarrow\HomT{T'}\).
Since  \(\overline g\) and \(r\) are both right almost split, there are morphisms \[a:\rad_{\Gamma}(T,T')\rightarrow \Hom_{\mathcal T}(T,E) \textnormal{ and }a':\Hom_{\mathcal T}(T,E)\rightarrow \rad_{\Gamma}(T,T')\] such that \(\overline g a=r\) and \(ra'=\overline g\). Hence
\[r a' a=\overline g a=r.\]
Since \(r\) is a monomorphism, \(a'a=1_{\rad_{\Gamma}(T,T')}\). Thus \(\rad_{\Gamma}(T,T')\) must be a direct summand of \(\HomT E\). We rewrite in terms of this direct summand: \(\HomT E=\rad_{\Gamma}(T,T')\oplus \Hom_{\mathcal T}(T,U)\) for some object \(U\) in \(\mathcal T\). We rewrite the morphism \(\overline g=\sm{r&  u}\)

We have
\[\sm{r&  u}= \sm{r&  u}\circ a \circ a'=\sm{r& 0 }.\]
Hence \(u=0\).

Let \(R\in\! \mathcal T\) be the preimage of \(\rad_{\Gamma}(T,T')\), i.e.\ \(\Hom_{\mathcal T}(T,\!R)=\rad_{\Gamma}(T,T')\). Let \(g'\) be the preimage of \(r\), so that \(r=\overline{g'}\). Since \(\Hom_{\mathcal T}(T,R)\) is a summand of \(\Hom_{\mathcal T}(T,E)\) it is clear that \(R\) is a summand of \(E\), that is \(E=R\oplus V\) where \(\Hom_{\mathcal T}(T,V)=\Hom_{\mathcal T}(T,U)\). Hence \(\Delta\) can be written as 
\[ \tau T'\xrightarrow{\sm{f_R \\ f_V}}R\oplus V\xrightarrow{\sm{g' & 0}} T'\rightarrow\Sigma\tau T' \]
which, by applying \(\HomT -\), is sent to the long exact sequence
\[\cdots \rightarrow\HomT {\tau T'}\xrightarrow{\sm{\overline{f_R}\\\overline{f_V}}} \begin{matrix}\HomT R\\ \oplus\\ \HomT V\end{matrix}\xrightarrow{\sm{\overline{g'}& 0}}\HomT{T'}\rightarrow \cdots\]

We know that \(\overline{g'}\) is a monomorphism. Exactness of the sequence gives \(\overline{g'}\circ\Hom_{\mathcal T}(T,f_R)=0\) and hence \(\Hom_{\mathcal T}(T,f_R)=0\). Due to lemma \ref{fromTenough} we also have \(f_R=0\). 

We have 
\[\bm{0&0\\0& 1_V}\bm{0\\f_V}=\bm{0\\f_V}.\]
Since \(f=\sm{0\\f_V}\) is left minimal, \(\sm{0&0\\0& 1_V}\) must be an automorphism on \(R\oplus V\), so \(R=0\). Then \(\rad\HomT{T'}=\HomT R=0\), and \(\HomT{T'}\) is a simple projective.

Consider \(\Delta\) under \(\HomT -\):
\[\HomT{\Sigma^{-1} T'}\rightarrow\HomT{\tau T}\rightarrow\HomT{E}\xrightarrow{0}\HomT{T'}\]
The first morphism cannot be zero, as \(\tau T\ncong E\). By Lemma \ref{proj0}, we must have \(\Sigma^{-1} T'\in\add T\). By induction, for any \(n\in\mathbb N\), we have \(\Sigma^{-n}T'\in\add T\). In particular \(\Sigma^{-2}T'=\Sigma^{-1}\tau^- T'\in\add T\). Hence, by Lemma \ref{inj0}, \(\HomT{T'}\) is injective. Thus \(\HomT{T'}\) is simple, projective and injective as a \(\End_\mathcal T(T)^\op\)-module.

We assumed \(\mathcal T\), and thus also \(\modf \Gamma\), to be a connected category. However, if \(\HomT{T'}\) is a simple, projective and injective module, it must be the only indecomposable object in its connected component. It follows that \(\modf\Gamma=\modf k\).
\end{proof}

\bibliographystyle{plain}

\end{document}